\documentclass[a4paper]{article}
\usepackage{amsmath,amsthm,amssymb,enumerate,float,amscd,layout,mathrsfs,color}
\newtheorem{dfn}{Definition}[section]
\newtheorem{thm}[dfn]{Theorem}

\newtheorem{lem}[dfn]{Lemma}

 \makeatletter
 \@addtoreset{equation}{section} 
\makeatother
\newcommand{\N}{\mathbb{N}}

\newcommand{\R}{\mathbb{R}}

\newcommand{\A}{\mathcal{A}}

\newcommand{\Hom}{\mathop{\mathrm{Hom}}\nolimits}

\def\Xint#1{\mathchoice
 {\XXint\displaystyle\textstyle{#1}}%
 {\XXint\textstyle\scriptstyle{#1}}%
 {\XXint\scriptstyle\scriptscriptstyle{#1}}%
 {\XXint\scriptscriptstyle\scriptscriptstyle{#1}}%
 \!\int} \def\XXint#1#2#3{{\setbox0=\hbox{$#1{#2#3}{\int}$}
 \vcenter{\hbox{$#2#3$}}\kern-.5\wd0}} 
\def\dashint{\Xint{-}}
\makeatletter
 \def\@maketitle{%
\begin{flushright}%
{\large \@date}%
\end{flushright}%
\par\vskip 1.5em
\begin{center}%
{\LARGE \@title \par}%
\end{center}%
\begin{center}%
{\large \@author}%
\end{center}%
\par\vskip 1.5em
 }
\title{Partial regularity for parabolic systems with VMO-coefficients}
\author{Taku Kanazawa \\
Graduate School of Mathematics \\
Nagoya University, JAPAN
}
\date{Ver. Dec. 17 2013 
}
\begin{document}
\maketitle

\begin{center}
\begin{minipage}[t]{10cm}
\small{
\noindent \textbf{Abstract.}
We establish a partial H\"older continuity for vector-valued solutions $u:\Omega_T\to\R^N$ to parabolic systems
of the type:
\[
 u_t-\mathrm{div}\bigl( A(x,t,u,Du)\bigr) = H(x,t,u,Du) \qquad \mathrm{in}\>\Omega\times (-T,0),
\]
where the coefficients $A:\Omega\times(-T,0)\times\R^N\times\Hom(\R^n,\R^N)\to\Hom(\R^n,\R^N)$ are possibly discontinuous 
with respect to $(x,t)$. More precisely, we assume a VMO-condition with respect to $(x,t)$ and continuity with respect to $u$ 
and prove H\"older continuity of the solutions outside of singular sets.
\medskip

\noindent \textbf{Keywords.} Nonlinear parabolic systems, Partial regularity, VMO-coefficients, $\A$-caloric approximation.
\medskip

\noindent \textbf{Mathematics~Subject~Classification~(2010):}
35K40, 35K55, 35B65.

}
\end{minipage}
\end{center}

\section{Introduction}
In this paper, we establish a partial regularity result of weak solutions to second order nonlinear parabolic systems 
of the following type:
\begin{equation}
 u_t-\mathrm{div}\bigl( A(z,u,Du)\bigr) = H(z,u,Du), \qquad z=(x,t)\in\Omega\times (-T,0)=:\Omega_T ,
 \label{system}
\end{equation}
where $\Omega$ denotes a bounded domain in $\R^n$, $n\geq 2$, $T>0$, $u$ takes values in $\R^N$, $N\geq 1$, and the vector field 
$A\colon\Omega_T\times\R^N\times\mathrm{Hom}(\R^n,\R^N)\to\Hom(\R^n,\R^N)$ fulfills the $p$-growth condition, 
$p\geq 2$, and the VMO-condition. More precisely, we assume that the partial mapping
$z\mapsto A(z,u,w)/(1+|w|)^{p-1}$ has vanishing mean oscillation (VMO), uniformly in $(u,w)$. This means that
$A$ satisfies the estimate
\[
 |A(z,u,w)-(A(\cdot,u,w))_{z_0,\rho}|\leq V_{z_0}(z,\rho)(1+|w|)^{p-1},
\]
where $V_{z_0}:\R^{n+1}\times[0,\rho_0]\to[0,2L]$ are bounded functions with 
\[
 \lim_{\rho\searrow 0}V(\rho)=0, \quad 
 V(\rho):=\sup_{z_0\in\Omega_T}\sup_{0<r\leq\rho}\dashint_{Q_r(z_0)\cap\Omega_T}V_{z_0}(z,r)dz.
\]
The vector field $A$ also satisfies the $p$-growth condition such as 
\[
 \left\lvert A(z,u,w)\right\rvert+(1+\lvert w\rvert)
 \left\lvert \partial_w A(z,u,w) \right\rvert\leq L(1+\lvert w\rvert)^{p-1}
\] 
for all $z\in\Omega_T$, $u\in\R^N$ and $w\in\Hom(\R^n,\R^N)$. Moreover $A$ is continuous with respect to $u$.
Roughly speaking, under the above assumptions, we prove 
that the bounded weak solutions of \eqref{system} are H\"older continuous on 
some open set $\Omega_u\subset\Omega_T$, i.e., $u\in C^{\alpha,\alpha/2}(\Omega_u,\R^N)$ (see Theorem \ref{pr}). 

Regularity problem of weak solutions to parabolic systems are already proved for nonlinear systems with $p=2$ by 
Duzaar-Mingione \cite{DM}, for $p\geq 2$ by Duzaar-Mingione-Steffen \cite{DMS}, 
for $1<p<2$ by Scheven \cite{Scheven} and even on the boundary by B\"ogelein-Duzaar-Mingione \cite{BDM1,BDM2}.  
These previous results are based on the technique 
so called ``$\A$-caloric approximation''(see Lemma \ref{A-caloric}) and proved under the condition that the vector field 
$A(z,u,w)$ are H\"older continuous with respect to $(z,u)$, i.e., there exists a non-decreasing function 
$K\colon [0,\infty)\to [1,\infty)$ and $\beta\in (0,1)$ such that the inequality 
\[
 \lvert A(z,u,w)-A(z_0,u_0,w)\rvert\leq K(\lvert u\rvert) 
 (\lvert x-x_0\rvert +\sqrt{\lvert t-t_0\rvert}+\lvert u-u_0\rvert)^\beta(1+\lvert w\rvert^{p-1})
\] 
holds for every $z=(x,t), z_0=(x_0,t_0)\in\Omega_T$, $u,u_0\in\R^N$ and for all $w\in\Hom(\R^n,\R^N)$. 

The $\A$-caloric approximation technique has its origin in the classical harmonic approximation lemma of De Giorgi in version of Simon
\cite{DeGiorgi1,Simon}. It was first applied to nonlinear elliptic systems with quadratic growth condition ($p=2$) 
by Duzaar-Grotowski \cite{DG}, namely ``$\A$-harmonic approximation''. 
Using this method, we could obtain the optimal regularity result without the reverse H\"older inequalities,  
i.e., if the ``coefficients'' $A(x,u,w)$ are H\"older continuous in $(x,u)$ with some
H\"older exponent $\beta\in (0,1)$ then $Du$ is H\"older continuous with the same exponent $\beta$ on some open set $\Omega_u$. 

Then the $\A$-harmonic approximation technique has been used to prove the regularity result for elliptic systems with 
super-quadratic growth ($p\geq 2$) and for the case of sub-quadratic growth ($1<p<2$) 
by Chen-Tan \cite{CT1,CT2}. The $\A$-harmonic approximation technique also 
works for boundary regularity which was proved by Grotowski \cite{Gr}. 
Moreover, a relation between the regularity of weak solutions and the smoothness of coefficients is studied. 
Duzaar-Gastel \cite{DGa} proved that weak solutions has $C^1$-regularity 
if the coefficients satisfies Dini-type condition (which is weaker assumption 
than H\"older continuity condition). The continuous coefficients would not ensure the continuity (and not even boundedness) of 
the gradient $Du$ but Foss-Mingione \cite{FM} showed that we could still except the local H\"older continuity of the solution $u$ itself . 
The H\"older continuity for the solution $u$ can also be guaranteed under discontinuous coefficients such as the VMO-condition 
in elliptic setting, which was proved for homogeneous systems by B\"ogelein-Duzaar-Habermann-Scheven \cite{BDHS} 
and for inhomogeneous systems by author \cite{Kana}. 

On the other hand, $\A$-harmonic approximation technique is adapted to parabolic systems, renamed as 
``$\A$-caloric approximation'' \cite{DM,DMS}, and it lead us to the partial regularity result for weak solutions in 
parabolic setting with H\"older continuous coefficients. Dini-type condition and the condition under continuous coefficients are 
also proved by Baroni \cite{Baroni}, B\"ogelein-Duzaar-Mingione \cite{BFM} and Foss-Geisbauer \cite{FG}. However, as far as we know, 
no one has been proved regularity result under discontinuous coefficients in parabolic systems. In this paper, we proved  
the regularity result under the VMO-condition which is the parabolic version of \cite{Kana} (see Theorem \ref{pr}).

\section{Statement of the results}
Before we start setting the structure conditions, let us collect some notations which we will use throughout the paper. 
As mentioned above, we consider a cylindrical domain $\Omega_T=\Omega\times(-T,0)$ where 
$\Omega$ is a bounded domain in $\R^n$, $n\geq 2$, and $T>0$. $u$ maps from $\Omega_T$ to $\R^N$, $N\geq 1$, and $Du$ denotes 
the gradient with respect to the special variables $x$, i.e., $Du(x,t)\equiv D_x u(x,t)$.  
We write $B_\rho(x_0):=\{ x\in\R^n\> :\> \lvert x-x_0\rvert <\rho\}$ and 
$Q_\rho(z_0):=B_\rho(x_0)\times (t_0-\rho^2, t_0)$ where $z_0=(x_0,t_0)\in\Omega_T$. The parabolic metric $d_{\mathrm{par}}$ is given by 
\begin{equation}
 d_{\mathrm{par}}(z,z_0)=\max\Bigl\{ \lvert x-x_0\rvert , \sqrt{\lvert t-t_0\rvert}\Bigr\} \qquad \text{for}\ 
 z=(x,t),z_0=(x_0,t_0)\in\Omega_T , \label{para-metric}
\end{equation}
and for a given set $X$ we denote by $\mathcal{H}^{n+2}_{\mathrm{par}}(X)$ the $(n+2)$-dimensional parabolic Hausdorff measure which is 
defined by 
\[
 \displaystyle\mathcal{H}^{n+2}_{\mathrm{par}}(X)=\sup_{\delta >0}\mathcal{H}^{n+2,\delta}_{\mathrm{par}}(X),
\]
where 
\[
 \displaystyle\mathcal{H}^{n+2,\delta}_{\mathrm{par}}(X)=\inf\left\{ \sum_{i=1}^\infty R_i^{n+2}\ :
 \ X\subset \bigcup_{i=1}^\infty Q_{R_i}(z_i),\ R_i\leq \delta\right\} .
\]
Note that $\mathcal{H}^{n+2}_{\mathrm{par}}$ is equivalent to the Lebesgue measure in $\R^{n+1}$, $\mathcal{L}^{n+1}$. 
For a bounded set $X\subset\R^{n+1}$ with $\mathcal{L}^{n+1}(X)>0$,
 we denote the average of a given function $g\in L^1(X,\R^N)$ by 
$\dashint_Xgdz$, that is, $\dashint_Xgdz =\frac{1}{\mathcal{L}^{n+1}(X)}\int_Xgdz$. In particular, we write
$g_{z_0,\rho}=\dashint_{Q_\rho(z_0)\cap\Omega}gdz$. We write $\mathrm{Bil}(\Hom(\R^n,\R^N))$
for the space of bilinear forms on the space $\Hom(\R^n,\R^N)$ of linear maps from $\R^n$ to
$\R^N$. We denote $c$ a positive constant, possibly varying from line by line. Special occurrences will be denoted by 
capital letters $K$, $C_1$, $C_2$ or the like.  

\begin{dfn}\label{wsol}
We say $u\in C^0(-T,0;L^2(\Omega,R^N))\cap L^p(-T,0;W^{1,p}(\Omega,\R^N))$, $p\geq 2$ 
is a weak solution of \eqref{system} if $u$ satisfies
\begin{equation}
 \int_{\Omega_T}\Bigl( \langle u,\varphi_t\rangle - \langle A(z,u,Du),D\varphi\rangle\Bigr) dz
 =\int_{\Omega_T} \langle H,\varphi\rangle dz \label{ws}
\end{equation}
for all $\varphi\in C^{\infty}_0(\Omega_T,\R^N)$, where $\langle\cdot,\cdot\rangle$ is the standard Euclidean 
inner product on $\R^N$ or $\R^{nN}$.
\end{dfn}

We assume the following structure conditions.
\begin{enumerate}
\item[({\bf H1})]
 $A(z,u,w)$ is differentiable in $w$ with continuous derivatives, that is, there exists
 $L\geq 1$ such that
\begin{equation}
 \left\lvert A(z,u,w)\right\rvert+(1+\lvert w\rvert)
 \left\lvert \partial_w A(z,u,w) \right\rvert\leq L(1+\lvert w\rvert)^{p-1}
\end{equation}
for all $z\in\Omega_T$, $u\in\R^N$ and $w\in\Hom(\R^n,\R^N)$. Moreover, from this we deduce the modulus of continuity 
function $\mu:[0,\infty)\to [0,\infty)$ such that $\mu$ is bounded, concave, non-decreasing and we have
\begin{equation}
 \left\lvert \partial_w A(z,u,w)-\partial_w A(z,u,w_0)\right\rvert
 \leq L\mu\left( \frac{\lvert w-w_0\rvert}{1+\lvert w\rvert+\lvert w_0\rvert}\right)
 (1+\lvert w\rvert+\lvert w_0\rvert)^{p-2}
\end{equation}
for all $z\in\Omega_T$, $u\in\R^N$, $w,w_0\in\Hom(\R^n,\R^N)$. Without loss of generality, we may assume $\mu\leq 1$.
\item[({\bf H2})]
 $A(z,u,w)$ is uniformly strongly elliptic, that is, for some $\lambda>0$ we have
\begin{equation}
 \biggl\langle \partial_w A(z,u,w)\tilde{w},\tilde{w} \biggr\rangle
 :=\sum_{\substack{1\leq i,\beta\leq N\\ 1\leq j,\alpha\leq n}}
 \partial_{w_\beta^j}A_\alpha^i(z,u,w)\tilde{w}_i^\alpha\tilde{w}_j^\beta
 \geq \lambda \lvert\tilde{w}\rvert^2(1+\lvert w\rvert^2)^{(p-2)/2}
\end{equation}
for all $z\in\Omega_T$, $u\in\R^N$, $w,\tilde{w}\in\Hom(\R^n,\R^N)$. 
\item[({\bf H3})]
 $A(z,u,w)$ is continuous with respect to $u$. There exists a bounded, concave and non-decreasing function 
$\omega:[0,\infty)\to [0,\infty)$ satisfying  
\begin{equation}
 \lvert A(z,u,w)-A(z,u_0,w)\rvert\leq L\omega\left( \lvert u-u_0\rvert^2\right)
 (1+\lvert w\rvert)^{p-1}
\end{equation}
for all $z\in\Omega_T$, $u,u_0\in\R^N$, $w\in\Hom(\R^n,\R^N)$. Without loss of generality, we may assume $\omega\leq 1$.
\item[({\bf H4})]
 $z\mapsto A(z,u,w)/(1+\lvert w\rvert)^{p-1}$ fulfils the following VMO-condition uniformly in $u$ and $w$:
\[
 \lvert A(z,u,w)- \left( A(\cdot ,u,w)\right)_{z_0 ,\rho} \rvert \leq
 V_{z_0}(z,\rho)(1+\lvert w\rvert)^{p-1}, \qquad \text{for all }z\in Q_{\rho}(z_0)
\]
whenever $z_0\in\Omega_T$, $0<\rho<\rho_0$, $u\in\R^N$ and $w\in\Hom(\R^n,\R^N)$, where $\rho_0>0$ and 
$V_{z_0}:\R^n\>\times\> [0,\rho_0]\to[0,2L]$ are bounded functions satisfying 
\begin{equation}
 \lim_{\rho\searrow 0}V(\rho)=0, \qquad V(\rho):=\sup_{z_0\in
 \Omega_T}\sup_{0<r\leq\rho}\dashint_{Q_r(z_0)\cap\Omega}V_{z_0}(z,r)dz.
\end{equation}
\item[({\bf H5})]
 $H(z,u,w)$ has $p$-growth, that is, there exist constants $a,b\geq 0$, with $a$ possibly depending on $M>0$,
such that 
\begin{equation}
 \lvert H(z,u,w)\rvert\leq a(M)\lvert w\rvert^p +b
\end{equation}
for all $z\in\Omega_T$, $u\in\R^N$ with $\lvert u\rvert\leq M$ and $w\in\Hom(\R^n,\R^N)$.
\end{enumerate}

Under these structure conditions, we proved the following theorem. 

\begin{thm}\label{pr}
Let $u\in C^0_b(-T,0;L^2(\Omega,R^N))\cap L^p(-T,0;W^{1,p}(\Omega,\R^N))$ be a bounded weak solution of the parabolic system 
\eqref{system} under the structure condition {\rm ({\bf H1}), ({\bf H2}), ({\bf H3}), ({\bf H4})} and {\rm ({\bf H5})} 
with satisfying $\| u\|_{\infty}\leq M$ and $2^{(10-9p)/2}\lambda>a(M)M$. Then there exists an open set $\Omega_u\subset\Omega_T$ 
such that $u\in C^{\alpha, \alpha/2}(\Omega_u, \R^N)$ with $\mathcal{H}^{n+2}_{\mathrm{par}}(\Omega_T\setminus\Omega_u)=0$ 
for every $\alpha\in (0,1)$. 
Moreover, $\Omega_T\setminus\Omega_u\subset\Sigma_1\cup\Sigma_2$ and 
\begin{align*}
 \Sigma_1 &:=\left\{ z_0\in\Omega_T\> :\> \liminf_{\rho\searrow 0}\dashint_{Q_\rho(z_0)}\lvert Du-(Du)_{z_0,\rho}\rvert^p dz>0\right\}, \\ 
 \Sigma_2 &:=\left\{ z_0\in\Omega_T\> :\> \limsup_{\rho\searrow 0}\lvert (Du)_{z_0,\rho}\rvert =\infty\right\} .
\end{align*}
\end{thm}
The previous result means that the weak solution $u$ is H\"older continuous in $\Omega_u$ with exponent $\alpha$ with respect to 
the parabolic metric given in \eqref{para-metric}. In other word, $u$ is H\"older continuous in $\Omega_u$ with exponent $\alpha$ 
with respect to space variable $x$ and with exponent $\alpha/2$ with respect to the time variable $t$. 

\section{Preliminaries}

In this section we present the $\A$-caloric approximation lemma and some standard estimates for the proof of our main theorem, 
(Theorem \ref{pr}). 

First we state the definition of $\A$-caloric function and recall the $\A$-caloric approximation lemma as below. 

\begin{dfn}[{$\A$-caloric function, \cite[DEFINITION 3.1]{DMS}}]
Let $\A$ be a bilinear form with constant coefficients satisfying 
\begin{equation}
 \lambda \lvert \tilde{w}\rvert^2\leq \A (\tilde{w},\tilde{w}), \qquad \A (w,\tilde{w})\leq L\lvert w\rvert\lvert\tilde{w}\rvert 
 \qquad \text{for all}\ w,\tilde{w}\in\Hom(\R^n,\R^N). \label{bil-property}
\end{equation}
A function $h\in L^2(t_0-\rho^2,t_0;W^{1,2}(B_\rho(x_0),\R^N))$ is called $\A$-caloric in the cylinder $Q_\rho(z_0)$ iff it satisfies 
\[
 \int_{Q_\rho(z_0)}\Bigl( \langle h,\varphi_t\rangle-\A(Dh,D\varphi)\Bigr) dz=0 
 \qquad \text{for all}\ \varphi\in C^\infty_0(Q_\rho(z_0),\R^N). 
\]
\end{dfn}
 
\begin{lem}[{$\A$-caloric approximation lemma, \cite[LEMMA 3.2]{DMS}}]\label{A-caloric}
Given $\varepsilon>0$, $0<\lambda<L$ and $p\geq 2$ there exists $\delta=\delta(n,N,p,\lambda,L,\varepsilon)\geq 1$ with the following 
property: Whenever $\A$ is a bilinear form on $\R^{nN}$ satisfying \eqref{bil-property}, $\gamma\in (0,1]$, and whenever 
\[
 w\in L^p(t_0-(\rho/2)^2,t_0; W^{1,2}(B_{\rho/2}(x_0),\R^N))
\]
is a function satisfying 
\begin{equation}
 \dashint_{Q_{\rho/2}(z_0)}\left(\left\lvert\frac{w}{\rho/2}\right\rvert^2 +\gamma^{p-2}\left\lvert\frac{w}{\rho/2}\right\rvert^p\right) dz 
 +\dashint_{Q_{\rho/2}(z_0)}\Bigl(\lvert Dw\rvert^2+\gamma^{p-2}\lvert Dw\rvert^p\Bigr) dz \leq 1
\end{equation}
and 
\begin{equation}
 \left\lvert\int_{Q_{\rho/2}(z_0)}\Bigl( \langle w,\varphi_t\rangle -\A(Dw,D\varphi) \Bigr) dz\right\rvert 
 \leq \delta \sup_{Q_{\rho/2}(z_0)}\lvert D\varphi\rvert
\end{equation}
for every $\varphi\in C^\infty_0(Q_{\rho/2}(z_0),\R^N)$ then there exists a function 
\[
 h\in L^p(t_0-(\rho/4)^2,t_0; W^{1,2}(B_{\rho/4}(x_0),\R^N))
\]
which is $\A$-caloric on $Q_{\rho/4}(z_0)$ such that 
\begin{equation}
 \dashint_{Q_{\rho/4}(z_0)}\left(\left\lvert\frac{h}{\rho/4}\right\rvert^2+\gamma^{p-2}\left\lvert\frac{h}{\rho/4}\right\rvert^p\right) dz 
 +\dashint_{Q_{\rho/4}(z_0)}\Bigl(\lvert Dh\rvert^2+\gamma^{p-2}\lvert Dh\rvert^p\Bigr) dz\leq 2\cdot 2^{n+2+2p}
\end{equation}
and 
\begin{equation}
 \dashint_{Q_{\rho/4}(z_0)}\left(\left\lvert\frac{w-h}{\rho/4}\right\rvert^2+\gamma^{p-2}\left\lvert\frac{w-h}{\rho/4}\right\rvert^p\right) dz 
 \leq \varepsilon.
\end{equation}
\end{lem}

The next lemma features a standard estimate for $\A$-caloric functions. 
 
\begin{lem}[{\cite[LEMMA 4.7]{DMS}}]\label{A-caloric-apriori}
Let $h\in L^2(t_0-(\rho/4)^2,t_0;W^{1,2}(B_{\rho/4}(x_0),\R^N))$ be $\A$-caloric function in $Q_{\rho/4}(z_0)$ with $\A$ satisfying 
\eqref{bil-property}. Then $h$ is smooth in $B_{\rho/4}(x_0)\times (t_0-(\rho/4)^2,t_0]$ and for any $s\geq 1$ there exists a 
constant $c_2=c_2(n,N,L/\lambda,s)\geq 1$ such that for any affine function $\ell :\R^n\to\R^N$ there holds 
\[
 \dashint_{Q_{\theta\rho}(z_0)}\left\lvert\frac{h-\ell}{\theta\rho}\right\rvert^s dz 
 \leq c_2\theta^2\dashint_{Q_\rho/4(z_0)}\left\lvert\frac{h-\ell}{\rho/4}\right\rvert^s dz 
 \qquad \text{for every}\ 0<\theta \leq 1/4. 
\]
\end{lem}

For given $u\in L^2(Q_\rho(z_0),\R^N)$ we denote by $\ell_{z_0,\rho}$ the unique affine function minimizing 
\begin{equation}
 \ell\mapsto\dashint_{Q_\rho(z_0)}\lvert u-\ell\rvert^2 dz \label{affine}
\end{equation}
among all affine functions $\ell(z)=\ell(x)$ which are independent of $t$. An elementary calculation yield that $\ell_{z_0,\rho}$ 
takes the form 
\[
 \ell_{z_0,\rho}(x)=\ell_{z_0,\rho}(x_0)+D\ell_{z_0,\rho}(x-x_0), 
\]
where 
\[
 \ell_{z_0,\rho}(x_0)=u_{z_0,\rho},\quad\text{and}\quad D\ell_{z_0,\rho}=\frac{n+2}{\rho^2}\dashint_{Q_\rho(z_0)}u\otimes (x-x_0)dz. 
\]
Using the Cauchy-Schwarz inequality we have the following lemma. 

\begin{lem}[{\cite[LEMMA 2.1]{DMS}}]\label{affine-esti}
Let $u\in L^2(Q_\rho(z_0),\R^N)$, $0<\theta <1$ and 
\[
 \ell_{z_0,\rho}(x)= \xi_{z_0,\rho}+D\ell_{z_0,\rho}(x-x_0), \qquad 
 \ell_{z_0,\theta\rho}(x)= \xi_{z_0,\theta\rho}+D\ell_{z_0,\theta\rho}(x-x_0)
\]
be the unique affine function that minimize 
\[
 \ell \mapsto \dashint_{Q_\rho(z_0)}\lvert u-\ell\lvert^2 dz \quad \text{and}\quad 
 \ell \mapsto \dashint_{Q_\theta\rho(z_0)}\lvert u-\ell\lvert^2 dz
\]
among all affine functions $\ell(z)=\ell(x)$ which are independent of $t$, respectively. Then there holds 
\begin{equation} 
 \lvert D\ell_{z_0,\theta\rho}-D\ell_{z_0,\rho}\rvert^2 
 \leq \frac{n(n+2)}{(\theta\rho)^2}\dashint_{Q_{\theta\rho}(z_0)}\lvert u-\ell_{z_0,\rho}\rvert^2 dz. \label{affine-esti1}
\end{equation}
Moreover, for any $D\ell\in \Hom(\R^n,\R^N)$ we have 
\begin{equation} 
 \lvert D\ell_{z_0,\rho}-D\ell\rvert^2 
 \leq \frac{n(n+2)}{\rho^2}\dashint_{Q_{\rho}(z_0)}\lvert u-u_{z_0,\rho}-D\ell (x-x_0)\rvert^2 dz. \label{affine-esti2}
\end{equation}
\end{lem}

Next two lemmas can also be obtained by elementary calculation. 

\begin{lem}[{\cite[Lemma 3.7]{Kana}}]\label{Young2}
Consider fixed $a,b\geq 0$, $p\geq 1$. Then for any $\varepsilon>0$, there exists
$K=K(p,\varepsilon)\geq 0$ satisfying
\begin{equation}
 (a+b)^p\leq (1+\varepsilon)a^p+Kb^p. \label{young2}
\end{equation}
\end{lem}
\begin{lem}[{\cite[Lemma 2.1]{GMo}}]\label{GM}
For $\delta \geq 0$, and for all $a,b\in\R^k$ we have 
\begin{equation}
 4^{-(1+2\delta)}\leq
 \frac{\displaystyle\int_0^1(1+|sa+(1-s)b|^2)^{\delta/2}ds}{(1+|a|^2+|b-a|^2)^{\delta/2}}
 \leq 4^\delta . \label{GM2}
\end{equation}
\end{lem}
\section{Proof of the main theorem}
To prove the regularity result (Theorem \ref{pr}), we first prove Caccioppoli-type inequality. 
In the followings, we define $q>0$ as 
the dual exponent of $p\geq 2$, that is, $q=p/(p-1)$. Here we note that $q\leq 2$. 
\begin{lem}\label{Caccioppoli}
Let $u\in C^0_b(-T,0;L^2(\Omega,R^N))\cap L^p(-T,0;W^{1,p}(\Omega,\R^N))$ be a bounded weak solution of the parabolic system 
\eqref{system} under the structure condition {\rm ({\bf H1}),({\bf H2}),({\bf H3}),({\bf H4})} and {\rm ({\bf H5})} 
with satisfying $\| u\|_{\infty}\leq M$ and $2^{(10-9p)/2}\lambda>a(M)M$. 
For any $ z_0=(x_0,t_0)\in\Omega_T$ and $\rho\leq 1$ with $Q_{\rho}(z_0)\Subset\Omega_T$, 
and any affine functions $\ell:\R^n\to\R^N$ with $\lvert\ell(x_0)\rvert\leq M$, we have 
the estimate
\begin{align}
 \sup_{t_0-(\rho/2)^2<t<t_0}&\dashint_{B_{\rho/2}(x_0)}\frac{\lvert u-\ell\rvert^2}{\rho^2(1+\lvert D\ell\rvert)^2} dx
 +\dashint_{Q_{\rho/2}(z_0)}\left\{\frac{\lvert Du-D\ell\rvert^2}{(1+\lvert D\ell\rvert)^2}
 +\frac{\lvert Du-D\ell\rvert^p}{(1+\lvert D\ell\rvert)^p}\right\} dz \notag\\
 \leq C_1&\Bigg[ \dashint_{Q_\rho(z_0)}\left\{
 \frac{\lvert u-\ell\rvert^2}{\rho^2(1+\lvert D\ell\rvert)^2}+
 \frac{\lvert u-\ell\rvert^p}{\rho^p(1+\lvert D\ell\rvert)^p}\right\} dz \notag\\
 &+\omega \left(\dashint_{Q_\rho(z_0)}\lvert u-\ell(x_0)\rvert^2dz \right)
 +V(\rho)+\left( a^q\lvert D\ell\rvert^q+b^q \right) \rho^q \Bigg], \label{caccioppoli}
\end{align}
with the constant $C_1=C_1(\lambda,p,L,a(M),M)\geq 1$. 
\end{lem}
\begin{proof}
Assume $z_0\in\Omega_T$ and $\rho\leq 1$ satisfy $Q_{\rho}(z_0)\Subset\Omega_T$. 
We take a standard cut-off functions $\chi\in C^\infty_0(B_\rho(x_0))$ and $\zeta\in C^1(\R)$. More precisely, let 
us take $\tilde{t}\in(t_0-\rho^2/4,t_0)$ and $\eta\in (0,\rho^2/4-\tilde{t}\,)$ and then $\zeta\in C^1(\R)$ satisfying 
\begin{equation}
 \left\{
 \begin{array}{ll}
  \zeta \equiv 1, & \text{on}\ (-\rho^2/4,\tilde{t}-\eta), \\ 
  \zeta \equiv 0, & \text{on}\ (-\infty,-\rho^2)\cup (\tilde{t},\infty), \\
  0\leq \zeta \leq 1, & \text{on}\ \R, \\ 
  \zeta_t =-1/\eta, & \text{on}\ (\tilde{t}-\eta , \tilde{t}), \\ 
  \lvert\zeta_t\rvert \leq 1/\rho^2, & \text{on}\ (-\rho^2,-\rho^2/4). 
 \end{array}\right.
 \label{cut-off}
\end{equation}
Moreover, $\chi\in C^\infty_0(B_\rho(x_0))$ satisfies $0\leq\chi\leq 1$, $\lvert D\chi\rvert\leq 4/\rho$, $\chi\equiv 1$ on $B_{\rho/2}(x_0)$.
Then $\varphi(x,t) :=\chi(t)\zeta(x)^p\bigl( u(x,t)-\ell(x)\bigr)$ is admissible as a test function in \eqref{ws}, and we obtain
\begin{align}
 \dashint_{Q_\rho(z_0)}\zeta\chi^p\langle &A(z,u,Du),Du-D\ell \rangle dz \notag\\
 =-\,&\dashint_{Q_\rho(z_0)}\langle A(z,u,Du),p\zeta\chi^{p-1}
 D\chi\otimes (u-\ell)\rangle dz \notag\\
 +&\dashint_{Q_\rho(z_0)}\langle u,\partial_t\varphi\rangle dz
 +\dashint_{Q_\rho(z_0)}\langle H,\varphi\rangle dz. \label{system2}
\end{align}
Furthermore, we have
\begin{align}
 -\,\dashint_{Q_\rho(z_0)}\zeta\chi^p\langle &A(z,u,D\ell),Du-D\ell \rangle dz
\notag\\
 =\,&\dashint_{Q_\rho(z_0)}\langle A(z,u,D\ell),p\zeta\chi^{p-1} D\chi\otimes
 (u-\ell)\rangle dz-\dashint_{Q_\rho(z_0)}\langle A(z,u,D\ell),D\varphi
 \rangle dz, \label{13}
\end{align}
and 
\begin{equation}
 \dashint_{Q_\rho(z_0)}\langle \left( A(\cdot ,\ell(x_0)
 ,D\ell)\right)_{z_0,\rho},D\varphi \rangle dz=0. \label{constcoeff}
\end{equation}
Adding \eqref{system2}, \eqref{13} and \eqref{constcoeff}, we obtain
\begin{align}
&\dashint_{Q_\rho(z_0)}\zeta\chi^p \langle A(z,u,Du)-A(z,u,D\ell),
 Du-D\ell\rangle dz \notag \\
 =&-\dashint_{Q_\rho(z_0)}\langle A(z,u,Du)-A(z,u,D\ell),
 p\zeta\chi^{p-1} D\chi\otimes (u-\ell)\rangle dz \notag \\
 &-\dashint_{Q_\rho(z_0)}\langle
A(z,u,D\ell)-A(z,\ell(x_0),D\ell),D\varphi\rangle dz
 \notag \\
 &-\dashint_{Q_\rho(z_0)}\langle A(z,\ell(x_0),D\ell)-\left(
 A(\cdot,\ell(x_0),D\ell)\right)_{z_0,\rho},D\varphi \rangle dz \notag \\
 &+\dashint_{Q_\rho(z_0)}\langle u-\ell ,\partial_t\varphi\rangle dz \notag\\ 
 &+\dashint_{Q_\rho(z_0)}\langle H,\varphi\rangle dz \notag \\
 =:& \>\> \hbox{I}+\hbox{II}+\hbox{III}+\hbox{IV}+\hbox{V}. \label{caccio-divide}
\end{align}
The terms $\hbox{I},\hbox{II},\hbox{III},\hbox{IV},\hbox{V}$ are defined above. 
Using the ellipticity condition ({\bf H2}) to the left-hand side of \eqref{caccio-divide}, we get
\begin{align}
 &\langle A(z,u,Du)-A(z,u,D\ell),Du-D\ell\rangle \notag\\
 =&\int_0^1\left\langle \partial_w A(z,u,sDu+(1-s)D\ell)(Du-D\ell),
 Du-D\ell\right\rangle ds \notag\\
 \geq& \lambda |Du-D\ell|^2\int_0^1(1+|sDu+(1-s)D\ell|)^{p-2}ds. \label{elliptic}
\end{align}
Then by using \eqref{GM2} in Lemma \ref{GM}, we obtain
\begin{align}
 &\langle A(z,u,Du)-A(z,u,D\ell),Du-D\ell\rangle \notag \\
 \geq& \lambda |Du-D\ell|^2\int_0^1
 (1+|sDu+(1-s)D\ell|^2)^{(p-2)/2}ds \notag\\
 \geq& 2^{(12-9p)/2}\lambda
 \left\{(1+|D\ell|)^{p-2}|Du-D\ell|^2+|Du-D\ell|^p\right\} . \label{elliptic2}
\end{align}
For $\varepsilon >0$ to be fixed later, using ({\bf H1}) and Young's inequality, we have
\begin{align}
 |\,\hbox{I}\,|
 \leq &\dashint_{Q_\rho(z_0)}p\zeta\chi^{p-1}\left\lvert\int_0^1\partial_w A(z,u,D\ell
 +s(Du-D\ell))(Du-D\ell)ds\right\rvert \lvert D\chi\rvert\lvert u-\ell\rvert dz \notag\\
 \leq &\dashint_{Q_\rho(z_0)}
 c(p,L)\zeta\chi^{p-1}\left\{ (1+\lvert D\ell\rvert )^{p-2}
 +\lvert Du-D\ell\rvert^{p-2}\right\} \lvert Du-D\ell\rvert \lvert D\chi\rvert\lvert u-\ell\rvert dz \notag\\
 \leq &\varepsilon\dashint_{Q_\rho(z_0)}\zeta\chi^p\left\{(1+\lvert D\ell\rvert)^{p-2}
 \lvert Du-D\ell\rvert^2+\lvert Du-D\ell\rvert^p\right\}dz \notag\\
 &+c(p,L,\varepsilon)\dashint_{Q_\rho(z_0)}\left\{
 (1+\lvert D\ell\rvert)^{p-2}\left\lvert\frac{u-\ell}{\rho}\right\rvert^2
 +\left\lvert\frac{u-\ell}{\rho}\right\rvert^p\right\}dz. \label{16}
\end{align}
In order to estimate $\hbox{II}$, we use ({\bf H3}), $D\varphi=\zeta\chi^p(Du-D\ell)+p\zeta\chi^{p-1}D\chi\otimes
(u-\ell)$, and again Young's inequality, we get 
\begin{align}
 \lvert\,\hbox{II}\,\rvert
 \leq &\varepsilon\dashint_{Q_\rho(z_0)}\zeta\chi^p(1+\lvert D\ell\rvert)^{p-2}|Du-D\ell|^2dz
 +\varepsilon^{-1}\dashint_{Q_\rho(z_0)}L^2
 \omega^2\left(\lvert u-\ell(x_0)\rvert^2\right)(1+\lvert D\ell\rvert)^p dz \notag\\
 &+\varepsilon\dashint_{Q_\rho(z_0)}(1+\lvert D\ell\rvert)^{p-2}\left\lvert\frac{u-\ell}{\rho}\right\rvert^2dz
 +\varepsilon^{-1}\dashint_{Q_\rho(z_0)}(4Lp)^2
 \omega^2\left(\lvert u-\ell(x_0)\rvert^2\right)(1+\lvert D\ell\rvert)^p dz \notag\\
 \leq &\varepsilon\dashint_{Q_\rho(z_0)}\zeta\chi^p(1+\lvert D\ell\rvert)^{p-2}\lvert Du-D\ell\rvert^pdz
 +\varepsilon\dashint_{Q_\rho(z_0)}(1+\lvert D\ell\rvert)^{p-2}\left\lvert\frac{u-\ell}{\rho}\right\rvert^pdz \notag\\
 &+c(p,L,\varepsilon)(1+\lvert D\ell\rvert)^p \omega^2\left(\dashint_{Q_\rho(z_0)}\lvert u-\ell(x_0)\rvert^2dz\right), \label{17}
\end{align}
where we use Jensen's inequality in the last inequality. 
We next estimate $\hbox{III}$ by using the VMO-condition ({\bf H4}) and Young's inequality, we have
\begin{align*}
 \lvert\hbox{III}\rvert
 \leq & \frac{\varepsilon}{2^{p-1}}\dashint_{Q_\rho(z_0)}
 \left\{\zeta\chi^p\lvert Du-D\ell\rvert+\frac{4p\zeta\lvert u-\ell\rvert}{\rho}\right\}^pdz
 +\left(\frac{2^{p-1}}{\varepsilon}\right)^{q/p}\dashint_{Q_\rho(z_0)}
 {V_{z_0}}^q(x,\rho)(1+\lvert D\ell\rvert)^pdz.
\end{align*}
Then using the fact that ${V_{z_0}}^q={V_{z_0}}^{q-1}\cdot V_{z_0}\leq (2L)^{q-1}V_{z_0}\leq 2LV_{z_0}$, 
we infer
\begin{align}
 \lvert\hbox{III}\rvert\leq & \varepsilon\dashint_{Q_\rho(z_0)}\zeta\chi^p\lvert Du-D\ell\rvert^pdz
 +c(p,\varepsilon)\dashint_{Q_\rho(z_0)}\left\lvert\frac{u-\ell}{\rho}\right\rvert^pdz
 +c(p,L,\varepsilon)(1+\lvert D\ell\rvert)^pV(\rho). \label{18}
\end{align} 
To estimate $\hbox{IV}$, recall that $\zeta_t$ satisfies $\zeta_t=-1/\eta$ on $(\tilde{t}-\eta,\tilde{t})$ and 
$\lvert\zeta_t\rvert\leq 1/\rho^2$ on $(-\rho^2,-\rho^2/4)$. This implies 
\begin{align}
 \hbox{IV}
 =&\dashint_{Q_\rho(z_0)}\zeta_t\chi^p\lvert u-\ell\rvert^2 dz 
  +\dashint_{Q_\rho(z_0)}\zeta\chi^p\cdot\partial_t\frac{1}{2}\lvert u-\ell\rvert^2 dz \notag\\ 
 =&\frac{1}{2}\dashint_{Q_\rho(z_0)}\zeta_t\chi^p\lvert u-\ell\rvert^2 dz \notag\\ 
 =&\frac{1}{2\lvert Q_\rho(z_0)\rvert}\int_{t_0-\rho^2}^{t_0-\rho^2/4}\int_{B_\rho(x_0)}\chi^p\left\lvert\frac{u-\ell}{\rho}\right\rvert^2 dxdt 
  -\frac{1}{2\eta\lvert Q_\rho(z_0)\rvert}\int_{\tilde{t}-\eta}^{\tilde{t}}\int_{B_\rho(x_0)}\chi^p\lvert u-\ell\rvert^2 dxdt \notag\\
 \leq &\frac{1}{2}\dashint_{Q_\rho(z_0)}(1+\lvert D\ell\rvert)^{p-2}\left\lvert\frac{u-\ell}{\rho}\right\rvert^2 dz 
  -\frac{1}{2\eta\lvert Q_\rho(z_0)\rvert}\int_{\tilde{t}-\eta}^{\tilde{t}}\int_{B_\rho(x_0)}\chi^p\lvert u-\ell\rvert^2 dz. 
 \label{para-term}
\end{align}
For $\varepsilon'>0$ to be fixed later, using ({\bf H5}), Lemma \ref{Young2} and Young's inequality, we have
\begin{align}
&|\,\hbox{V}\,| \notag\\
 \leq & \dashint_{Q_\rho(z_0)}a(\lvert Du-D\ell\rvert+\lvert D\ell\rvert)^p\zeta\chi^p\lvert u-\ell\rvert dz
 +\dashint_{Q_\rho(z_0)}(b\zeta\chi^p\rho)\left\lvert\frac{u-\ell}{\rho}\right\rvert dz \notag\\
 \leq & \dashint_{Q_\rho(z_0)}a\zeta\chi^p\left\{(1+\varepsilon')\lvert Du-D\ell\rvert^p
 +K(p,\varepsilon')\lvert D\ell\rvert^p\right\}\lvert u-\ell\rvert dz
 +\varepsilon b^q\rho^q
 +\varepsilon^{-p/q}\dashint_{Q_\rho(z_0)}\left\lvert\frac{u-\ell}{\rho}\right\rvert^p dz \notag\\
 \leq & a(1+\varepsilon')(2M+\lvert D\ell\rvert\rho)
 \dashint_{Q_\rho(z_0)}\zeta\chi^p\lvert Du-D\ell\rvert^p dz+ 
 c(p,\varepsilon)\dashint_{Q_\rho(z_0)} \left\lvert\frac{u-\ell}{\rho}\right\rvert^p dz \notag\\
 &+\varepsilon(1+\lvert D\ell\rvert)^p\rho^q\left\{ a^qK ^q \lvert D\ell\rvert^{q}+b^q\right\}. \label{19}
\end{align}
Combining \eqref{caccio-divide}, \eqref{elliptic2}, \eqref{17}, \eqref{18}, \eqref{para-term} and \eqref{19}, and set
$\lambda'=2^{(12-9p)/2}\lambda$C
$\Lambda :=\lambda'-3\varepsilon-a(1+\varepsilon')(2M+\lvert D\ell\rvert\rho)$, this gives
\begin{align}
 &\frac{1}{2\eta}\int_{\tilde{t}-\eta}^{\tilde{t}}
 \dashint_{B_\rho(x_0)}\chi^p\left\lvert\frac{u-\ell}{\rho(1+\lvert D\ell\rvert)}\right\rvert^2 dz
 +\Lambda\dashint_{Q_\rho(z_0)}\zeta\chi^p
 \left\{\frac{\lvert Du-D\ell\rvert^2}{(1+\lvert D\ell\rvert)^2}
 +\frac{\lvert Du-D\ell\rvert^p}{(1+\lvert D\ell\rvert)^p}\right\}dz \notag\\
 \leq & c(p,L,\varepsilon)\left[\dashint_{Q_\rho(z_0)}
 \left\{\left\lvert\frac{u-\ell}{\rho(1+\lvert D\ell\rvert)}\right\rvert^2
 +\left\lvert\frac{u-\ell}{\rho(1+\lvert D\ell\rvert)}\right\rvert^p\right\}dz
 +\omega\left(\dashint_{Q_\rho(z_0)}|u-\ell(x_0)|^2dz\right) +V(\rho) \right] \notag\\
 &+\varepsilon\left\{ a^q(1+K(p,\varepsilon'))^q|D\ell|^q+b^q\right\}\rho^q. \label{roughcaccio}
\end{align}
Now choose $\varepsilon=\varepsilon(\lambda ,p,a(M),M)>0$ and $\varepsilon'=\varepsilon'(\lambda ,p,a(M),M)>0$ in a right way 
(for more precise way of choosing $\varepsilon$ and $\varepsilon'$, we refer to \cite[Lemma 4.1]{DG}) and taking the limit 
$\eta\to 0$, we obtain \eqref{caccioppoli}. 
\end{proof}

To use the $\A$-caloric approximation lemma, we need to estimate 
$\dashint_{Q_\rho(z_0)}((u-\ell)\cdot\varphi_t-\A(D(u-\ell),D\varphi)) dz$.
\begin{lem}\label{A-caloric2}
Assume the same assumptions in Lemma \ref{Caccioppoli}. Then for any $z_0=(x_0,t_0)\in\Omega_T$ 
and $\rho\leq \rho_0$ satisfy $Q_{2\rho}(z_0)\Subset\Omega_T$, and any affine functions $\ell:\R^n\to\R^N$ 
with $|\ell(x_0)|\leq M$, the inequality
\begin{align}
 &\dashint_{Q_\rho(z_0)}\Bigl( \langle v,\varphi_t\rangle -\mathcal{A}(Dv,D\varphi)\Bigr) dz \notag\\
 \leq &C_2(1+\lvert D\ell\rvert)\biggl[ \mu^{1/2}\left(\sqrt{\Psi_*(z_0,2\rho,\ell)}\right)
 \sqrt{\Psi_*(z_0,2\rho,\ell)} 
 +\Psi_*(z_0,2\rho,\ell)+\rho(a\lvert D\ell\rvert^p+b)\biggr]
 \sup_{Q_\rho(z_0)}\lvert D\varphi\rvert \label{Ah}
\end{align}
holds for all $\varphi\in C^\infty_0(B_\rho(x_0),\R^N)$ and a constant $C_2=C_2(n,\lambda,L,p,a(M))\geq 1$, where 
\begin{align*}
 \A(Dv,D\varphi):&=\frac{1}{(1+\lvert D\ell\rvert)^{p-1}}\left\langle
 \left( \partial_w A(\cdot,\ell(x_0),D\ell)\right)_{z_0,\rho}
 Dv,D\varphi\right\rangle , \\
 \Phi(z_0,\rho,\ell):&=\dashint_{Q_\rho(z_0)}
 \left\{\frac{\lvert Du-D\ell\rvert^2}{(1+\lvert D\ell\rvert)^2}
 +\frac{\lvert Du-D\ell\rvert^p}{(1+\lvert D\ell\rvert)^p}\right\} dz, \\
 \Psi(z_0,\rho,\ell):&=\dashint_{Q_\rho(z_0)}
 \left\{\frac{\lvert u-\ell\rvert^2}{\rho^2(1+\lvert D\ell\rvert)^2}
 +\frac{\lvert u-\ell\rvert^p}{\rho^p(1+\lvert D\ell\rvert)^p}\right\} dz, \\
 \Psi_*(z_0,\rho,\ell):&=\Psi(z_0,\rho,\ell)
 +\omega\left(\dashint_{Q_\rho(z_0)}\lvert u-\ell(x_0)\rvert^2dz\right)+V(\rho)
 +\left( a^q\lvert D\ell\rvert^q+b^q\right)\rho^q, \\
 v:&=u-\ell=u-\ell(x_0)-D\ell(x-x_0).
\end{align*}
\end{lem}
\begin{proof}
Assume $z_0\in\Omega_T$ and $\rho\leq 1$ satisfy $Q_{2\rho}(z_0)\Subset\Omega_T$.
Without loss of generality we may assume $\displaystyle\sup_{Q_\rho(z_0)}\lvert D\varphi\rvert\leq 1$.
Note $\displaystyle\sup_{Q_\rho(z_0)}\lvert\varphi\rvert\leq\rho\leq 1$. Using the fact that 
$\int_{Q_\rho(z_0)}A(z_0,\ell(x_0),w)D\varphi dx=0$, we deduce
\begin{align}
 &(1+\lvert D\ell\rvert)^{p-1}\dashint_{Q_\rho(z_0)}\Bigl( v\cdot\varphi_t-\mathcal{A}(Dv,D\varphi)\Bigr) dz \notag\\
 =&\dashint_{Q_\rho(z_0)}\int_0^1
 \left\langle\left[\bigl( \partial_w A(\cdot,\ell(x_0),D\ell)\bigr)_{z_0,\rho}
 -\bigl( \partial_w A(\cdot,\ell(x_0),D\ell+sDv)\bigr)_{z_0,\rho}
 \right] Dv,D\varphi\right\rangle dsdz \notag\\
 &+\dashint_{Q_\rho(z_0)}
 \left\langle\bigl(A(\cdot,\ell(x_0),Du)\bigr)_{z_0,\rho}-A(z,\ell(x_0),Du),
 D\varphi\right\rangle dz \notag\\
 &+\dashint_{Q_\rho(z_0)}
 \langle A(z,\ell(x_0),Du)-A(z,u,Du),D\varphi\rangle dz \notag\\
 &+\dashint_{Q_\rho(z_0)}\langle H,\varphi\rangle dz \notag\\
 =&:\hbox{I}+\hbox{II}+\hbox{III}+\hbox{IV} \label{21}
\end{align}
where terms $\hbox{I},\hbox{II},\hbox{III},\hbox{IV}$ are define above. 

Using the modulus of continuity $\mu$ from ({\bf H1}), Jensen's inequality and H\"{o}lder's inequality, we estimate 
\begin{align}
 \lvert\,\hbox{I}\,\rvert
 &\leq 
 c(p,L)(1+\lvert D\ell\rvert)^{p-1}\dashint_{Q_\rho(z_0)}\mu\left(\frac{\lvert Du-D\ell\rvert}{1+\lvert D\ell\rvert}\right)
 \left\{\frac{\lvert Du-D\ell\rvert}{1+\lvert D\ell\rvert}+\frac{\lvert Du-D\ell\rvert^{p-1}}{(1+\lvert D\ell\rvert)^{p-1}}\right\}dz \notag\\
 &\leq 
 c\, (1+\lvert D\ell\rvert)^{p-1}\left[ \mu^{1/2}\left(\sqrt{\Phi(z_0,\rho,\ell)}\right)\sqrt{\Phi(z_0,\rho,\ell)}
 +\mu^{1/p}\left(\Phi^{1/2}(z_0,\rho,\ell)\right)\Phi^{1/q}(z_0,\rho,\ell)\right] \notag\\
 &\leq 
 c\, (1+\lvert D\ell\rvert)^{p-1}\left[ \mu^{1/2}\left(\sqrt{\Phi(z_0,\rho,\ell)}\right) \sqrt{\Phi(z_0,\rho,\ell)}
 +\Phi(z_0,\rho,\ell)\right]. \label{22}
\end{align}
The last inequality follows from the fact that $a^{1/p}b^{1/q}=a^{1/p}b^{1/p}b^{(p-2)/p}\leq a^{1/2}b^{1/2}+b$ 
holds by Young's inequality. 

By using the VMO-condition, Young's inequality and the bound $V_{x_0}(x,\rho)\leq 2L$, 
the term $\hbox{II}$ can be estimated as 
\begin{align}
 \lvert\,\hbox{II}\,\rvert
 &\leq 
 c(p)(1+\lvert D\ell\rvert)^{p-1}\dashint_{Q_\rho(z_0)}
 \left\{V_{z_0}(z,\rho)+V_{z_0}(z,\rho)\frac{\lvert Du-D\ell\rvert^{p-1}}{(1+\lvert D\ell\rvert)^{p-1}}\right\}dz \notag\\
 &\leq 
 c\, (1+\lvert D\ell\rvert)^{p-1} \left[\left( 1+(2L)^{p-1}\right) V(\rho)+\Phi(z_0,\rho,\ell)\right] . \label{23}
\end{align}
Similarly, we estimate the term $\hbox{III}$ by using the continuity condition ({\bf H3}), Young's inequality, the bound
$\omega\leq 1$ and Jensen's inequality. This leads us to 
\begin{align}
 \lvert\hbox{III}\rvert
 &\leq L\dashint_{Q_\rho(z_0)}
 (1+\lvert D\ell\rvert+\lvert Du-D\ell\rvert)^{p-1}\omega\left(\lvert u-\ell(x_0)\rvert^2\right)dz \notag\\
 &\leq 
 c(p,L)(1+\lvert D\ell\rvert)^{p-1} \left[\omega\left(\dashint_{Q_\rho(z_0)}\lvert u-\ell(x_0)\rvert^2dz\right)
 +\Phi(z_0,\rho,\ell)\right] .  \label{24}
\end{align}
By using the growth condition ({\bf H5}) and $\displaystyle\sup_{B_\rho(x_0)}\lvert\varphi\rvert\leq\rho\leq 1$, we have
\begin{align}
 \lvert\hbox{IV}\rvert
 &\leq \dashint_{Q_\rho(z_0)}\rho(a\lvert Du\rvert^p+b)dz \notag\\
 &\leq 2^{p-1}a(1+\lvert D\ell\rvert)^p\Phi(z_0,\rho,\ell)
 +2^{p-1}\rho(1+\lvert D\ell\rvert)^{p-1}(a\lvert D\ell\rvert^p+b). \label{25}
\end{align}
Therefore combining \eqref{21}, \eqref{22}, \eqref{23}, \eqref{24} and \eqref{25}, and using Caccioppoli-type inequality 
(Lemma \ref{Caccioppoli}), we have 
\begin{align*}
 &\dashint_{Q_\rho(z_0)}\left( \langle v,\varphi_t\rangle -\A(Dv,D\varphi)\right) dz \\
 \leq &2^{p+1}(1+\lvert D\ell\rvert)^p(1+a+(2L)^{p-1}) \\
  &\times\left[ \mu^{1/2}\left(\sqrt{\Phi(z_0,\rho,\ell)}\right)\sqrt{\Phi(z_0,\rho,\ell)}
  +\Phi(z_0,\rho,\ell)+\Psi_*(z_0,\rho,\ell)+\rho(a\lvert D\ell\rvert^p+b)\right] \\
 \leq &C_2(1+\lvert D\ell\rvert)^p
  \left[ \mu^{1/2}\left(\sqrt{\Psi_*(z_0,2\rho,\ell)}\right)\sqrt{\Psi_*(z_0,2\rho,\ell)}
  +\Psi_*(z_0,2\rho,\ell)+\rho(a\lvert D\ell\rvert^p+b)\right] ,
\end{align*}
where we set $C_2:=2^{n+p+3}C_1(1+a+(2L)^{p-1})$ at the last inequality and this completes the proof.
\end{proof}

From now on, we write $\Phi(\rho)=\Phi(z_0,\rho,\ell_{z_0,\rho})$, $\Psi(\rho)=\Psi(z_0,\rho,\ell_{z_0,\rho})$, 
$\Psi_*(\rho)=\Psi_*(z_0,\rho,\ell_{z_0,\rho})$ for $z_0\in\Omega_T$ and $0<\rho\leq 1$. Here $\ell_{z_0,\rho}$ is a minimizer 
which we introduce in \eqref{affine}. 

Now we are ready to establish the excess improvement. 

\begin{lem}\label{Ex-imp}
Assume the same assumption in Lemma \ref{Caccioppoli}. 
Let $\theta\in (0,1/4]$ be arbitrary and impose the following smallness conditions on the excess: 
\begin{enumerate}
 \item $\mu^{1/2}\left(\sqrt{\Psi_*(\rho)}\right)+\sqrt{\Psi_*(\rho)}\leq \frac{\delta}{2}$ with the constant 
  $\delta =\delta(n,N,p,\lambda,L,\theta^{n+p+4})$ from Lemma \ref{A-caloric}, 
 \item $\Psi(\rho)\leq \frac{\theta^{n+4}}{4n(n+2)}$, 
 \item $\gamma(\rho):=[{\Psi_*}^{q/2}(\rho)+\delta^{-q}\rho^q(a\lvert D\ell\rvert+b)^q]^{1/q}\leq 1$.
\end{enumerate}
Then there holds the excess improvement estimate 
\begin{equation}
 \Psi(\theta\rho)\leq C_3\theta^2\Psi_*(\rho) \label{ex-imp}
\end{equation}
with a constant $C_3=C_3(n,\lambda,L,p,a(M))\geq 1$.
\end{lem}
\begin{proof}
Set 
\[
 w:=\frac{u-\ell_{x_0,\rho}}{C_2(1+\lvert D\ell\rvert)\gamma(\rho)}. 
\]
From Lemma \ref{A-caloric2} and the assumption (i) we have 
\begin{align*}
 \dashint_{Q_{\rho/2}(z_0)}\Bigl( \langle w,\varphi_t\rangle -\mathcal{A}(Dw,D\varphi)\Bigr) dz 
 \leq &\biggl[ \mu^{1/2}\left(\sqrt{\Psi_*(\rho)}\right) +\sqrt{\Psi_*(\rho)} 
 +\frac{\delta}{2}\biggr]\sup_{Q_{\rho/2}(z_0)}\lvert D\varphi\rvert \notag\\
 \leq &\delta\sup_{Q_{\rho/2}(z_0)}\lvert D\varphi\rvert ,
\end{align*}
for all $\varphi\in C^\infty_0(Q_{\rho/2}(z_0),\R^N)$. Moreover, using Caccioppoli-type inequality and the assumption (iii), we get 
\begin{align*}
 &\dashint_{Q_{\rho/2}(z_0)}\left\lvert\frac{w}{\rho/2}\right\rvert^2+\gamma^{p-2}\left\lvert\frac{w}{\rho/2}\right\rvert^p dz 
 +\dashint_{Q_{\rho/2}(z_0)}\lvert Dw\rvert^2+\gamma^{p-2}\lvert Dw\rvert^p dz \\ 
 \leq &\frac{1}{C_2^2\gamma^2}\left\{ 2^{n+p+2}\Psi(\rho)+C_1\Psi_*(\rho)\right\} \\
 \leq &\frac{\max\{ 2^{n+p+2},C_1\}}{C_2^2}\leq 1.
\end{align*}
Therefore the $\A$-caloric approximation lemma (Lemma \ref{A-caloric}) implies the existence of  
\[
 h\in L^p(t_0-(\rho/4)^2,t_0; W^{1,2}(B_{\rho/4}(x_0),\R^N))
\]
which is $\A$-caloric on $Q_{\rho/4}(z_0)$ and satisfies  
\[
 \dashint_{Q_{\rho/4}(z_0)}\left(\left\lvert\frac{h}{\rho/4}\right\rvert^2+\gamma^{p-2}\left\lvert\frac{h}{\rho/4}\right\rvert^p\right) dz 
 +\dashint_{Q_{\rho/4}(z_0)}\Bigl(\lvert Dh\rvert^2+\gamma^{p-2}\lvert Dh\rvert^p\Bigr) dz\leq 2\cdot 2^{n+2+2p}
\]
and 
\begin{equation}
 \dashint_{Q_{\rho/4}(z_0)}\left(\left\lvert\frac{w-h}{\rho/4}\right\rvert^2+\gamma^{p-2}\left\lvert\frac{w-h}{\rho/4}\right\rvert^p\right) dz 
 \leq \theta^{n+p+4}. \label{A-caloric-esti}
\end{equation}
Then from Lemma \ref{A-caloric-apriori}, we have for $s=2$ respectively for $s=p$
\begin{align*}
 &\gamma^{s-2}\left(\theta\rho\right)^{-s}
  \dashint_{Q_{\theta\rho}(z_0)}\lvert h-h_{z_0,\rho/4}-(Dh)_{z_0,\rho/4}(x-x_0)\rvert^s dz \\ 
 \leq &c(s)\gamma^{s-2}\theta^s\left(\frac{\rho}{4}\right)^{-s} 
  \dashint_{Q_{\rho/4}(z_0)}\lvert h-h_{z_0,\rho/4}-(Dh)_{z_0,\rho/4}(x-x_0)\rvert^s dz \\ 
 \leq &3^{s-1}c(s)\gamma^{s-2}\theta^s\left(\frac{\rho}{4}\right)^{-s} 
  \left[\dashint_{Q_{\rho/4}(z_0)}\lvert h\rvert^s dz+\lvert h_{z_0,\rho/4}\rvert^s 
  +\lvert (Dh)_{z_0,\rho/4}\rvert^s\left(\frac{\rho}{4}\right)^s\right] \\ 
 \leq &2\cdot 3^{s-1}c(s)\gamma^{s-2}\theta^s\left[\left(\frac{\rho}{4}\right)^{-s} 
  \dashint_{Q_{\rho/4}(z_0)}\lvert h\rvert^s dz+\dashint_{Q_{\rho/4}(z_0)}\lvert Dh\rvert^s dz\right] \\
 \leq &2^{n+4+p}\cdot 3^{s-1}c(s)\theta^s .
\end{align*}
Thus, using \eqref{A-caloric-esti} we obtain 
\begin{align*}
 &\gamma^{s-2}\left(\theta\rho\right)^{-s} 
  \dashint_{Q_{\theta\rho}(z_0)}\lvert w-h_{z_0,\rho/4}-(Dh)_{z_0,\rho/4}(x-x_0)\rvert^s dz \\ 
 \leq &2^{s-1}\left(\theta\rho\right)^{-s} 
  \left[ \dashint_{Q_{\theta\rho}(z_0)}\gamma^{s-2}\lvert w-h\rvert^s dz 
  +\gamma^{s-2}\dashint_{Q_{\theta\rho}(z_0)}\lvert h-h_{z_0,\rho/4}-(Dh)_{z_0,\rho/4}(x-x_0)\rvert^s dz \right] \\
 \leq &2^{s-1}
  \left[4^{n+2-s}\theta^{-n-2-s}\dashint_{Q_{\rho/4}(z_0)}\gamma^{s-2}\left\lvert\frac{w-h}{\rho/4}\right\rvert^s dz 
  +3^{s-1}\cdot 2^{n+4+p}c(s)\theta^s \right] \\
 \leq &2^{s-1}\Bigl( 4^{n+2-s}+3^{s-1}\cdot 2^{n+4+p}c(s)\Bigr)\theta^2 .
\end{align*}
Scaling back to $u$ we have 
\begin{align*}
 &(\theta\rho)^{-s}\dashint_{Q_{\theta\rho}(z_0)}\lvert u-\ell_{z_0,\theta\rho}\rvert^s dz \\ 
 \leq & c(n,s)(\theta\rho)^{-s}\dashint_{Q_{\theta\rho}(z_0)} 
  \lvert u-\ell_{z_0,\rho}-C_2\gamma(1+\lvert D\ell_{z_0,\rho}\rvert)(h_{z_0,\rho/4}-(Dh)_{z_0,\rho/4}(x-x_0))\rvert^s dz \\ 
 = & c(n,s)C_2^s\gamma^s(1+\lvert D\ell_{z_0,\rho}\rvert)^s(\theta\rho)^{-s} 
  \dashint_{Q_{\theta\rho}(z_0)}\lvert w-h_{z_0,\rho/4}-(Dh)_{z_0,\rho/4}(x-x_0)\rvert^s dz \\
 \leq & c(n,s,p,C_2)\gamma^2(1+\lvert D\ell_{z_0,\rho}\rvert)^s\theta^2 \\ 
 \leq & c(1+\lvert D\ell_{z_0,\rho}\rvert)^s\theta^2[{\Psi_*}^{q/2}(\rho)+2^{q/p}\delta^{-q}\Psi_*(\rho)]^{2/q} \\ 
 \leq & c(1+\lvert D\ell_{z_0,\rho}\rvert)^s\theta^2\Psi_*(\rho) .
\end{align*}
Here we want to replace the term $(1+\lvert D\ell_{z_0,\rho}\rvert)$ by $(1+\lvert D\ell_{z_0,\theta\rho}\rvert)$. To do this, 
using \eqref{affine-esti1} from Lemma \ref{affine-esti} and the assumption (ii), we have 
\begin{align*}
 \lvert D\ell_{z_0,\theta\rho}-D\ell_{z_0,\rho}\rvert^2 
 \leq &\frac{n(n+2)}{(\theta\rho)^2}\dashint_{Q_{\theta\rho}(z_0)}\lvert u-\ell_{z_0,\rho}\rvert^2 dz \\ 
 \leq &\frac{n(n+2)}{\theta^{n+4}\rho^2}\dashint_{Q_\rho(z_0)}\lvert u-\ell_{z_0,\rho}\rvert^2 dz \\ 
 \leq &\frac{n(n+2)}{\theta^{n+4}}(1+\lvert D\ell_{z_0,\rho}\rvert)^2\Psi(\rho) 
 \leq \frac{1}{4}(1+\lvert D\ell_{z_0,\rho}\rvert)^2 .
\end{align*} 
This yields 
\[
 1+\lvert D\ell_{z_0,\rho}\rvert \leq 2(1+\lvert D\ell_{z_0,\theta\rho}\rvert) .
\]
Thus we have 
\[
 (\theta\rho)^{-s}\dashint_{Q_{\theta\rho}(z_0)}\lvert u-\ell_{z_0,\theta\rho}\rvert^s dz 
 \leq c(1+\lvert D\ell_{z_0,\theta\rho}\rvert)^s\theta^2\Psi_*(\rho) ,
\]
and this immediately yields the claim. 
\end{proof}

Let fix an arbitrarily H\"older exponent $\alpha\in (0,1)$ and define the Campanato-type excess 
\[
 C_\alpha (z_0,\rho):=C_\alpha(\rho)=\rho^{-2\alpha}\dashint_{Q_\rho(z_0)}\lvert u-u_{z_0,\rho}\rvert^2 dz .
\]
Here we iterate the excess improvement estimate \eqref{ex-imp} and obtain the boundedness of two excess functional, $\Psi_*$ and 
$C_\alpha$. 

\begin{lem}\label{excess-imp}
Assume the same assumption in Lemma \ref{Caccioppoli}. 
For every $\alpha\in (0,1)$ there exist constants $\varepsilon_*, \kappa_*, \rho_*>0$ and $\theta_*\in (0,1/8]$ such that 
the conditions 
\begin{equation}
 \Psi(\rho)<\varepsilon_* \qquad \text{and}\qquad C_\alpha(\rho)<\kappa_* \tag{$A_0$}
\end{equation}
for all $0<\rho<\rho_*$ with $Q_\rho(z_0)\Subset\Omega_T$, imply 
\begin{equation}
 \Psi(\theta_*^k\rho)<\varepsilon_* \qquad \text{and}\qquad C_\alpha(\theta_*^k\rho)<\kappa_* \tag{$A_k$} \label{Ak}
\end{equation}
respectively, for every $k\in\N$. 
\end{lem}
\begin{proof}
First set 
\[
 \theta_*:=\min \left\{ \left(\frac{1}{16n(n+2)}\right)^{1/(2-2\alpha)}, \frac{1}{\sqrt{4C_3}}\right\} \leq \frac{1}{8}, 
\]
and take $\varepsilon_*>0$ which satisfies 
\[
 \varepsilon_*\leq \frac{\theta_*^{n+4}}{16n(n+2)} \qquad \text{and}\qquad 
 \mu^{1/2}\left(\sqrt{4\varepsilon_*}\right)+\sqrt{4\varepsilon_*}\leq \frac{\delta}{2}. 
\] 
Note that the choice of $\theta_*$ fixes the constant $\delta=\delta(n,N,\lambda,L,p,\theta_*^{n+p+4})>0$ from Lemma \ref{A-caloric}. 
Then choose $\kappa_*>0$ so small that 
\[
 \omega(\kappa_*)<\varepsilon_*. 
\]
Finally, we take $\rho_*>0$ which satisfies 
\[
 \rho_*\leq \min\{ \rho_0, \kappa_*^{1/(2-2\alpha)},1\}, \quad V(\rho_*)<\varepsilon_* \quad\text{and}\quad 
 \left\{\left(a\sqrt{n(n+2)\kappa_*}\right)^q+b^q\right\}\rho_*^{q\alpha}<\varepsilon_*. 
\]
Now we prove the assertion \eqref{Ak} by induction. First using \eqref{affine-esti2} from Lemma \ref{affine-esti} with 
$\ell\equiv u_{z_0,\theta^k\rho}$ and the assumption \eqref{Ak}, we obtain 
\begin{align}
 \lvert D\ell_{z_0,\theta^k\rho}\rvert^2 
 &\leq \frac{n(n+2)}{(\theta^k\rho)^2}\dashint_{Q_{\theta^k\rho}(z_0)}\lvert u-u_{z_0,\theta^k\rho}\rvert^2 dz \notag\\ 
 &\leq n(n+2)(\theta^k\rho)^{2-2\alpha}C_\alpha (z_0,\theta^k\rho) \notag\\ 
 &\leq n(n+2)\rho_*^{2-2\alpha}\kappa_* . \label{Dell-esti}
\end{align}
Thus, we have 
\begin{align*}
 \Psi_*(\theta^k\rho)\leq &\Psi(\theta^k\rho) +\omega(C_\alpha(z_0,\theta^\rho)) +V(\theta^k\rho)
  +(a^q\lvert D\ell_{z_0,\theta^\rho}\rvert^q +b^q)(\theta^k\rho)^q \\ 
 \leq &\varepsilon_* +\omega(\kappa_*) +V(\rho_*) +\left\{\left( a\sqrt{n(n+2)\kappa_*}\right)^q +b^q\right\}\rho_*^{q\alpha}
  <4\varepsilon_* . 
\end{align*}
This implies 
\begin{equation}
 \mu^{1/2}\left(\sqrt{\Psi_*(\theta^\rho)}\right) +\sqrt{\Psi_*(\theta^k\rho)}
 <\mu^{1/2} \left(\sqrt{4\varepsilon_*}\right) +\sqrt{4\varepsilon_*}\leq \frac{\delta}{2},  \label{smc1}
\end{equation} 
and 
\begin{equation}
 \Psi(\theta^k\rho) <\varepsilon_*< \frac{\theta^{n+4}}{4n(n+2)}. \label{smc2}
\end{equation}
Furthermore, we have 
\begin{equation}
 \gamma(\theta^k\rho) =\left[ \Psi_*^{q/2}(\theta^k\rho) +\delta^{-q}(\theta^k\rho)^q 
 (a\lvert D\ell_{z_0,\theta^k\rho}\rvert +b)^q\right]^{1/q}\leq 1.  \label{smc3}
\end{equation}
To check \eqref{smc3}, the first term of \eqref{smc3} can be estimated by the choice of $\varepsilon_*$ and the fact 
$\Psi_*(\theta^k\rho)<1$: 
\[
 \Psi_*^{q/2}(\theta^k\rho)\leq \Psi_*^{1/2}(\theta^k\rho)<\sqrt{4\varepsilon_*}\leq \frac{\delta}{2}. 
\]
To estimate the second term of \eqref{smc3}, using \eqref{Dell-esti} and the fact $\rho_*^{\alpha -1}\geq 1$, we obtain 
\begin{align*}
 \delta^{-q}(\theta^k\rho)(a\lvert D\ell_{z_0,\theta^k\rho}\rvert +b)^q 
 &\leq \delta^{-q}\rho_*^q\left( a\rho_*^{\alpha -1}\sqrt{n(n+2)\kappa_*}+b\right)^q \\ 
 &\leq \delta^{-q}\rho_*^{q\alpha}\left( a\sqrt{n(n+2)\kappa_*}+b\right)^q \\ 
 &\leq 2^{q/p}\delta^{-q}\rho_*^{q\alpha}\varepsilon_* \\ 
 &\leq 2^{-4+q/p}\delta^{2-q}\leq\frac{\delta}{8}. 
\end{align*}
Therefore, we have \eqref{smc3} and this allowed us to apply Lemma \ref{excess-imp} with the radius $\theta^k\rho$ instead of $\rho$, 
which yields 
\[
 \Psi(\theta^{k+1}\rho)\leq C_3\theta^2\Psi_*(\theta^k\rho) <4C_3\theta^2\varepsilon_* \leq \varepsilon_* .
\]
Thus, we have established the first part of the assertion $(A_{k+1})$ and it remains to prove the second one, that is, 
$C_\alpha(z_0,\theta^{k+1}\rho)$. For this aim, we first compute
\[
 \frac{1}{(\theta_*^k\rho)^2}\dashint_{Q_{\theta_*^k\rho}(z_0)}
 |u-\ell_{z_0,\theta_*^k\rho}|^2 dz
 \leq (1+|D\ell_{z_0,\theta_*^k\rho}|)^2\Psi(\theta_*^k\rho)
 \leq 2\varepsilon_*+2\varepsilon_*|D\ell_{z_0,\theta_*^k\rho}|^2
\]
where we used the assumption \eqref{Ak} in the last step. 
Since $\ell_{z_0,\theta_*^k\rho}(x)=u_{z_0,\theta_*^k\rho}+D\ell_{z_0,\theta_*^k\rho} (x-x_0)$, we can estimate
\begin{align*}
 C_\alpha(z_0,\theta_*^{k+1}\rho)&\leq
 (\theta_*^{k+1}\rho)^{-2\alpha}\dashint_{Q_{\theta_*^{k+1}\rho}(z_0)} |u-u_{z_0,\theta_*^k\rho}|^2dz \\
 &\leq 2(\theta_*^{k+1}\rho)^{-2\alpha}\left[\dashint_{Q_{\theta_*^{k+1}\rho}(z_0)}
 |u-\ell_{z_0,\theta_*^k\rho}|^2dz+|D\ell_{z_0,\theta_*^k\rho}|^2(\theta_*^{k+1}\rho)^2\right] \\
 &\leq 2(\theta_*^{k+1}\rho)^{-2\alpha}\left[\theta_*^{-n-2}\dashint_{Q_{\theta_*^k\rho}(z_0)}
 |u-\ell_{z_0,\theta_*^k\rho}|^2dz+|D\ell_{z_0,\theta_*^k\rho}|^2(\theta_*^{k+1}\rho)^2\right] \\
 &\leq 4(\theta_*^k\rho)^{2-2\alpha}\left[\varepsilon_*\theta_*^{-n-2-2\alpha}
 +|D\ell_{z_0,\theta_*^k\rho}|^2(\varepsilon_*\theta_*^{-n-2-2\alpha}+\theta_*^{2-2\alpha})\right] .
\end{align*}
Recalling the choice of $\rho_*$, $\varepsilon_*$ and $\theta_*$, we deduce
\begin{align*}
 C_\alpha(x_0,\theta^{k+1}\rho)
 &\leq 4{\rho_*}^{2-2\alpha}\left[\varepsilon_*\theta_*^{-n-2-2\alpha}+n(n+2)
 \kappa_*{\rho_*}^{2-2\alpha}(\varepsilon_*\theta_*^{-n-2-2\alpha}+\theta_*^{2-2\alpha})\right]\\
 &\leq\frac{1}{4}{\rho_*}^{2-2\alpha}\theta_*^{2-2\alpha}
 +8n(n+2)\kappa_*\theta_*^{2-2\alpha} \\
 &\leq\frac{1}{4}\kappa_*+\frac{1}{2}\kappa_*<\kappa_*.
\end{align*}
This proves the second part of the assertion $(A_{k+1})$ and we completes the proof. 
\end{proof}
To obtain the regularity result (Theorem \ref{pr}), it is similar arguments in \cite{BDHS} with using the integral characterization 
of H\"older continuous functions with respect to the parabolic metric of Campanato-Da Prato \cite{Prato}.  
\mbox{}\\

\def\cprime{$'$}

\providecommand{\bysame}{\leavevmode\hbox to3em{\hrulefill}\thinspace}

\providecommand{\MR}{\relax\ifhmode\unskip\space\fi MR }


\providecommand{\MRhref}[2]{%
  
\href{http://www.ams.org/mathscinet-getitem?mr=#1}{#2}
}

\providecommand{\href}[2]{#2}

\mbox{}\\
Taku Kanazawa \\
Graduate School of Mathematics \\
Nagoya University \\
Chikusa-ku, Nagoya, 464-8602, JAPAN \\
taku.kanazawa@math.nagoya-u.ac.jp
\end{document}